\documentclass{jT} 

\usepackage{amssymb,amsmath,latexsym,amsthm} 
\usepackage[english]{} 

\let\ds=\displaystyle 
 
  \def\Q{\mathbb{Q}} 
  \def\Z{\mathbb{Z}} 
  \def\F{\mathbb{F}} 
\def\ab{{\rm ab}} 
\def\No{{\rm N}} 

\def\plus{\ds\mathop{\raise 2.0pt \hbox{$\bigoplus $}}\limits} 
\def\mult{\ds\mathop{\raise 2.0pt \hbox{$\bigotimes$}}\limits} 
\def\prd{ \ds\mathop{\raise 2.0pt \hbox{$  \prod   $}}\limits} 
\def\Cap{ \ds\mathop{\raise 2.0pt \hbox{$\bigcap   $}}\limits} 
\def\sm{  \ds\mathop{\raise 2.0pt \hbox{$  \sum    $}}\limits}

\def\Cup{\displaystyle \mathop{\raise 2.0pt \hbox{$\bigcup$}}\limits}
\def\ov{\overline} 

\newtheorem{theo}{Theorem} 
\newtheorem{prop}{Proposition} 
\newtheorem{remark}{Remark} 

\begin{document} 

\title{\hspace{-0.8cm} Stickelberger's congruences \hspace{-0.8cm} \\ 
\hspace{-0.8cm} for absolute norms of
\hspace{-0.8cm} \\ \hspace{-0.8cm}  relative discriminants} 
 
\author[Georges {\sc Gras}]{{\sc Georges} GRAS} 

\address{Georges {\sc Gras}\\ 
Villa la Gardette, chemin Ch\^ateau Gagni\`ere,\\ 
F-38520 Le Bourg d'Oisans} 
\email{g.mn.gras@wanadoo.fr} 
\urladdr{http://monsite.orange.fr/maths.g.mn.gras/} 


\date{27 mars 2010, r\'evis\'e le 8 avril 2010} 



\maketitle 

\begin{resume} 
Nous g\'en\'eralisons un r\'esultat de J. Martinet 
sur les congruences de Stickelberger pour les normes absolues 
des discriminants relatifs des corps de nombres, 
en utilisant des arguments classiques du corps de classes. 
\end{resume} 

\begin{abstr} 
We give an improvement of a result of J. Martinet on 
Stickelberger${}'$s congruences for the absolute norms 
of relative discriminants of number fields, 
by using classical arguments of class field theory. 
\end{abstr} 

\section{Introduction} 

Let $L/K$ be a finite extension of number fields. 
Denote by ${\mathfrak d}_{L/K}$ the relative 
discriminant of $L/K$ and by $c$ the number of complex infinite places 
of $L$ which lie above a real place of $K$. 

\smallskip 
The absolute norm of an ideal ${\mathfrak a}$ of $K$ is a positive 
rational
denoted $\ov\No_{K/\Q}\, ({\mathfrak a})$; it is the positive generator of 
$\No_{K/\Q}\, ({\mathfrak a})$, where $\No_{K/\Q}$ 
is the arithmetic norm. If $\alpha \in K^\times$, we define the 
{\em absolute norm of $\alpha$} (or $(\alpha)$) by 
$\ov\No_{K/\Q}\, (\alpha) := \vert\,\No_{K/\Q}\, (\alpha)\,\vert$ 
(this has some importance in class field theory). 

\smallskip 
In [Ma], J. Martinet proved the following result 
about $\ov\No_{K/\Q}\, ({\mathfrak d}_{L/K})$: 

\begin{prop} If $K$ contains a primitive $2^{m+1}$th root of unity 
($m \geq 0$) and if $L/K$ is not ramified at $2$, 
then $(-1)^c\, \ov\No_{K/\Q}\, ({\mathfrak d}_{L/K}) \equiv 1 
\bmod (4\cdot 2^{m})$. 
\end{prop} 

In [Pi], S. Pisolkar proved, in connection with the previous result: 

\begin{prop} Let $p \geq 2$ be any prime number. 
    Let $K_v$ be the completion of $K$ at a place 
    $v \,\vert \, p$ (or any finite extension of $\Q_p$); 
    we suppose that $K_v$ contains 
    a primitive $p^{h+1}$th root of unity, $h \geq 0$. 
    Let $K_v(\sqrt[p]{\alpha\,})$, $\alpha \in     K_v^\times$, 
    be an unramified Kummer extension of $K_v$. 
Then $\No_{K_v/\Q_p}\,( \alpha) \equiv 1 \bmod (p^{h+2})$. 
\end{prop} 

\smallskip 
In this paper we give a synthetic proof of these results with some 
generalization of the hypothesis (especially for the case $p=2$); 
see Theorem~\ref{thmain}. 

\section{Prerequisites on discriminants} 
Classical proofs of Stickelberger's congruences make use of the fact 
that any odd discriminant ideal ${\mathfrak d}_{L/K}$ 
is canonically associated with the discriminant 
of a quadratic extension of $K$, unramified at~$2$. 
This essential reduction is summarized in the following proposition 
(see [Ma, \S\;3]). 

\begin{prop} Let $L/K$ be a finite extension of number fields 
and let $\alpha \,K^{\times 2}$, 
 in $K^\times/K^{\times 2}$, be the image of the discriminant $\alpha$ 
of a $K$-base of $L$. Then: 
 
\smallskip 
 (i) The class $\alpha \,K^{\times 2}$ does not depend on the 
 choice of the $K$-base.
 
\smallskip
 (ii) Let $K' := K(\sqrt{\alpha\,})$; then there exists an integral 
 ideal ${\mathfrak a}$ of $K$ such that 
 ${\mathfrak d}_{L/K} = {\mathfrak d}_{K'/K}\, {\mathfrak a}^2$. 
 
\smallskip 
 (iii) If $2$ is unramified in $L/K$ it is unramified in $K'/K$ 
 and we have ${\mathfrak d}_{K'/K} = (\alpha)\,{\mathfrak b}^2$, 
hence ${\mathfrak d}_{L/K} = (\alpha) \, {\mathfrak c}^2$, 
where ${\mathfrak b}$ and ${\mathfrak c}$ are ideals of $K$. 
\end{prop} 

We suppose in the sequel that ${\mathfrak d}_{L/K}$ is odd; 
thus we can choose, modulo $K^{\times 2}$, an odd $\alpha$, 
which implies that ${\mathfrak c}$ is odd. 
We then have to compute $\ov\No_{K/\Q} \,({\mathfrak c})^2$ 
and $\ov\No_{K/\Q}\, (\alpha)$. 

\section{Computation of $\ \ov\No_{K/\Q}\, ({\mathfrak c})^2$.} 
From class field theory over $\Q$ we get $\No_{K/\Q} \,({\mathfrak c}) \in 
A_K$, the Artin group of $K$ which is that of $K^\ab$, 
where $K^\ab$ is the maximal abelian subextension of $K$. 

So we see that to obtain nontrivial congruences modulo a power of~$2$ 
we must suppose that this Artin group is roughly a ray group mudulo 
a power of~$2$ in the following way. 

\medskip 
Let $\Q(\mu_{2^\infty})$ be the field generated by all roots of unity 
of order a power of 2. The best hypothesis is that $K$ does contain 
a subfield $k$ of $\Q(\mu_{2^\infty})$ of degree~$2^m$, $m\geq 0$. 

\smallskip 
For $m=0$ we get $k = \Q$ (which is also $\Q^{(0)}$ in the description 
below) and for any $m\geq 1$, the field $k$ is 
equal to one of the following three fields, for which we indicate 
its Artin group as a subgroup of 
$A_\Q := \{ u\,\Z,\ u\in \Q^\times, \ u\ {\rm odd}\,\}$ 
(see e.g. [Gr, II.5.5.2]): 

\smallskip 
$\ \ \bullet\ \ $ $k = \Q^{(m)}$ is the subfield, of degree $2^m$, 
of the cyclotomic $\Z_2$-extension of $\Q$; its Artin group is: 
$$A_{\Q^{(m)}} = \{ u\,\Z,\ u\in \Q^\times,\ u>0, \ u \equiv \pm 1 \bmod 
(4\cdot 2^m) \}\,; $$ 

\smallskip 
$\ \ \bullet\ \ $ $k = \Q'{}^{(m)}$, $m\geq 1$, is the subfield 
of $\Q(\mu_{4\cdot 2^m})$ 
of relative degree 2, distinct from $\Q(\mu_{4\cdot 2^{m-1}})$ 
and from $\Q^{(m)}$; its Artin group is 
$$A_{\Q'{}^{(m)}} = \{ u\,\Z,\ u\in \Q^\times,\ u>0, \ u 
\equiv 1\, \ {\rm or}\, -1+4\cdot 2^{m-1} \bmod (4\cdot 2^m) \}\,;$$ 

\smallskip 
$\ \ \bullet\ \ $ $k = \Q(\mu_{4\cdot 2^{m-1}})$, $m\geq 1$; 
its Artin group is 
$$A_{\Q(\mu_{4\cdot 2^{m-1}})} = \{ u\,\Z,\ u\in \Q^\times,\ u>0, \ u 
\equiv 1 \bmod (4\cdot 2^{m-1}) \}\,. $$ 

So this yields 
$$\No_{K/\Q}\, ({\mathfrak c})^2 \in \{ u\,\Z,\ u\in \Q^\times,\ u>0, 
\ u \equiv 1 \bmod (4\cdot 2^{m}) \}\,, $$ 
except if $k = \Q^{(m)}$, in which case 
$$\No_{K/\Q}\, ({\mathfrak c})^2 \in \{ u\,\Z,\ u\in \Q^\times,\ u>0, 
\ u \equiv 1 \bmod (4\cdot 2^{m+1})\}\,;$$ 
in other words, taking absolute norms: 
\begin{eqnarray*} 
\ov\No_{K/\Q}\, ( {\mathfrak c})^2 &\equiv& 1 \bmod (4\cdot 2^{m}), \ \ 
\hbox{if $k = \Q'{}^{(m)}$ or $\Q(\mu_{4\cdot 2^{m-1}})$, $m\geq 1$,}\\ 
\ov\No_{K/\Q}\, ( {\mathfrak c})^2 &\equiv& 1 \bmod (4\cdot 2^{m+1}),\ \ 
\hbox{ if $k = \Q^{(m)}$, $m\geq 0$.} 
\end{eqnarray*} 

\section{Computation of $\ \ov\No_{K/\Q}\, (\alpha)$.} 
The best way is to use local class field theory by computing 
$\No_{K/\Q}\, (\alpha) = \prd_{v\vert 2}\,\No_{K_v/\Q_2}\, (\alpha)$, 
where $K_v$ is the completion of $K$ at the place $v\,\vert\, 2$ 
of $K$ and $\No_{K_v/\Q_2}$ the local norm. 

\smallskip 
The result is given by the following generalization of the result of [Pi]. 

\medskip 
Let $p \geq 2$ be a prime number, let $K_v$ be the completion 
of the number field $K$ at $v\,\vert\,p$ (or any finite extension 
of $\Q_p$); we suppose that $K_v$ contains $\mu_p$ 
and a subfield $k_{(v)}$ of $\Q_p(\mu_{p^\infty})$ of degree $p^m$ 
over $\Q_p(\mu_p)$, $m \geq 0$ (if $p=2$, the context is 
that of the previous section for which the hypothesis are satisfied 
for all $v\,\vert\, 2$, 
with $k_{(v)} = k_v := \Q_2\,k$ and $[k_{(v)}:\Q_2]=[k:\Q]=2^m$, 
independently of $v\,\vert\, 2$, 
since $k/\Q$ is totally ramified at 2).%
\footnote{Take 
care that if $k = K \cap \Q(\mu_{2^\infty})$, $k_v$ may not be 
equal to $K_v \cap \Q_2(\mu_{2^\infty})$ (for instance 
$K = \Q(\sqrt{-17\,})$ for which $k=\Q$, $k_v = \Q_2$, $m=0$); 
but Theorem~\ref{thlocal} applies to 
$k_{(v)} = K_v \cap \Q_2(\mu_{2^\infty}) = 
\Q_2(\sqrt{-1\,})$ with $m=1$.} 

\smallskip 
The local norm group of 
$k_{(v)}$, restricted to the norms of units, is the following 
subgroup of $\Z_p^\times = \mu_{p-1} \oplus (1+ p\,\Z_p)$ 
for $p\ne 2$ or of 
\newline 
$\Z_2^\times = \langle\,-1\,\rangle \oplus (1 + 4\,\Z_2)$ for $p= 2$: 

\smallskip\smallskip 
$\ \ \bullet\ \ $ $p\ne 2$, $k_{(v)} = \Q_p(\mu_{p^{m+1}})$, $m\ge 0$; 
the norm group is $1 + p^{m+1}\,\Z_p$; 

\smallskip\smallskip 
$\ \ \bullet\ \ $ $p = 2$, $k_{(v)} = \Q_2^{(m)}$, $m\geq 0$; 
the norm group is 
$\langle\,-1\,\rangle \oplus (1 + 4\cdot 2^{m}\,\Z_2)$; 

\smallskip\smallskip 
$\ \ \bullet\ \ $ $p = 2$, $k_{(v)} = \Q'_2{}^{(m)}$, $m\geq 1$; 
the norm group is $\langle\,-1 + 4\cdot 2^{m-1} \,\rangle_{\Z_2}^{}$;


\smallskip\smallskip 
$\ \ \bullet\ \ $ $p = 2$, $k_{(v)} = \Q_2(\mu_{4\cdot 2^{m-1}})$, 
$m\geq 1$; the norm group is $1 + 4\cdot 2^{m-1}\,\Z_2$. 

\smallskip\smallskip 
We then have: 

\begin{theo}\label{thlocal} 
 Let $K_v$ be the completion 
of a number field $K$ at $v\,\vert\,p$; suppose that $K_v$ contains 
$\mu_p$ and a subfield $k_{(v)}$ of $\Q_p(\mu_{p^\infty})$ 
of degree $p^m$ over $\Q_p(\mu_p)$, $m \geq 0$. 
Let $K_v(\sqrt[p]{\alpha\,})$, $\alpha \in K_v^\times$, 
be an unramified Kummer extension of $K_v$ (modulo $K_v^{\times p}$
we can suppose that $\alpha$ is a local unit). 

Then we have 
$\,\No_{K_v/\Q_p}\, (\alpha) \equiv 1 \bmod (p^{m+2})$. 

Moreover, if $p=2$ and $k_{(v)} = \Q_2^{(m)}$, $m\geq 0$, 
and if at least one of the following two conditions holds: 

\smallskip 
\ \ (i) $\alpha \in K_v^{\times 2}$, 

\smallskip 
\ \ (ii) the index of ramification $e_v(K_v/k_{(v)})$ 
of $K_v/k_{(v)}$ is even, 

\smallskip\noindent 
we then have $\,\No_{K_v/\Q_2}\, (\alpha) \equiv 1 \bmod (4\cdot 2^{m+1})$. 
\end{theo} 

\begin{proof} We consider the following diagram: 

\unitlength=1.cm
$$\vbox{\hbox{\begin{picture}(12.1,3.2)
\put(8.35,3.0){\line(1,0){1.4}}
\put(4.5,3.0){\line(1,0){2.3}}
\put(7.8,1.5){\line(1,0){1.9}}

\put(8.5,0.0){\line(1,0){1.4}}

\put(4.5,1.5){\line(1,0){2.5}}
\put(2.,1.5){\line(1,0){1.5}}

\put(4.5,0.0){\line(1,0){2.3}}
\put(2.0,0.0){\line(1,0){1.55}}

\put(10.1,1.9){\line(0,1){0.75}}
\put(10.1,0.4){\line(0,1){0.75}}
\put(1.5,0.4){\line(0,1){0.75}}

\put(4.00,1.9){\line(0,1){0.75}}
\put(4.00,0.4){\line(0,1){0.75}}

\put(7.5,1.9){\line(0,1){0.75}}
\put(7.5,0.4){\line(0,1){0.75}}

\put(9.9,2.9){$K_v^{\rm nr}$}%
\put(6.9,2.9){$K_v(\sqrt[p]\alpha)$}%
\put(7.1,1.45){$\ \ \, .\ $}%
\put(3.7,2.9){$K_v$}%

\put(10.0,1.4){$k_{(v)}^{\rm nr}$}%
\put(1.3,1.4){$k_{(v)}$}%
\put(3.8,1.45){$\ .\ $}%

\put(3.8,-0.1){$F_v$}%
\put(6.94,-0.1){$F_v(\sqrt[p]{\alpha'\,})$}%
\put(0.8,-0.1){$\Q_p(\mu_p)$}%
\put(0.8,0.6){${}^{p^m}$}%
\put(10.0,-0.1){$\Q_p(\mu_p)^{\rm nr}$}%
\end{picture}   }} $$

\bigskip\noindent 
where for any field $L$, $L^{\rm nr}$ is the maximal 
unramified pro-extension of $L$; we know that we have 
for instance $L^{\rm nr} = L\,\Q_p^{\rm nr}$ (see e.g. [Gr, II.1.1.5]). 
Put $F_v := K_v \cap \Q_p(\mu_p)^{\rm nr}$. 
All horizontal extensions are unramified 
and all vertical extensions are totally ramified. 

\smallskip 
Consider the intersection 
$K_v(\sqrt[p]\alpha) \cap \Q_p(\mu_p)^{\rm nr}$ 
as a Kummer extension of $F_v$; 
thus there exists a suitable local unit $\alpha' \in F_v^\times$ 
such that $\alpha = \alpha'\, x^p$ with $x \in K_v^\times$. 

Then $\No_{K_v/\Q_p}\, (\alpha) = \No_{K_v/\Q_p}\, (\alpha')\cdot 
\No_{K_v/\Q_p}\,(x)^p$; since $F_v(\sqrt[p]{\alpha'\,})/F_v$ is 
\linebreak 
unramified we have (see e.g. [Gr, I.6.3, (ii)]): 
$$\alpha' = x'{}^p \,(1 + p\,(1-\zeta)\,y')\,,$$ 
$x',\,y' \in F_v^\times$, $v(y') \geq 0$, where 
$\zeta$ is a primitive $p$th root of unity 
(for $p=2$ we have $p\,(1-\zeta) = 4$). 
Then: 
$$\No_{K_v/\Q_p}\, (\alpha) = \No_{K_v/\Q_p}\, 
(1 + p\,(1-\zeta)\,y')\cdot\No_{K_v/\Q_p}\,(x\,x')^p\,;$$
but (see the above list of norm groups of $k_{(v)}$), we have 
$$\No_{K_v/\Q_p}\,(x\,x')^p \equiv 1 \bmod (p^{m+2})\,,$$ 
and in the particular case when $p=2$ and $k_{(v)} = \Q_2^{(m)}$, 
$$\No_{K_v/\Q_2}\,(x\,x')^2 \equiv 1 \bmod (4\cdot 2^{m+1})\,,$$ 
and 
$$\No_{K_v/\Q_p}\,(1 + p\,(1-\zeta)\,y') = \No_{F_v/\Q_p}\, 
(1 + p\,(1-\zeta)\,y')^{[K_v:F_v]}\equiv 1 \bmod (p^{m+2})$$ 
since $[K_v:F_v]$ is a multiple of $p^m$, 
which implies first that 
$$(1 + p\,(1-\zeta)\,y')^{[K_v:F_v]} = 1 + p^{m+1}\,(1-\zeta)\,y''\,,$$ 
and then that 
$$\No_{F_v/\Q_p}\,(1 + p^{m+1}\,(1-\zeta)\,y'') 
\in 1+p^{m+2}\,\Z_p\,;$$ 
hence we have the congruence 
$$\,\No_{K_v/\Q_p}\, (\alpha) \equiv 1 \bmod (p^{m+2})\,.$$ 

\smallskip 
\ \ (i) If $p=2$, $k_{(v)} = \Q_2^{(m)}$, $m \geq 0$, 
and $\alpha \in K_v^{\times 2}$, then we obtain
\linebreak 
$\No_{K_v/\Q_2}\, (\alpha) \equiv 1 \bmod (4\cdot 2^{m+1})$. 

\smallskip 
\ \ (ii) If $p=2$, $e_v(K_v/k_{(v)})$ is even, 
the above computation yields 
\linebreak 
$\No_{F_v/\Q_2}\, (1 + 4\,y')^{[K_v:F_v]}\equiv 1$ $\bmod \ 
(4\cdot 2^{m+1})$; if moreover $k_{(v)} = \Q_2^{(m)}$, 
since $\No_{K_v/\Q_2}\,(x\,x')^2 \equiv 1 \bmod (4\cdot 2^{m+1})$ 
in that case, we obtain 
$$\No_{K_v/\Q_2}\, (\alpha) \equiv 1 \bmod (4\cdot 2^{m+1})\,.$$ 
This completes the proof of the theorem. 
\end{proof} 

\section{Statement of the main result} 
We return to the case $p=2$. 
In Sections 3 and 4, we have computed 
$\ov\No_{K/\Q} \,({\mathfrak d}_{L/K})$, 
making use of $\ov\No_{K/\Q} \,({\mathfrak c})^2$ (absolute norm) 
and of $\No_{K/\Q} \, (\alpha)$ (arithmetic norm) from the 
$\No_{K_v/\Q_2}\, (\alpha)$,
taking into account that the congruence 
$\No_{K_v/\Q_2}\,(\alpha) \equiv 1 \bmod (4\cdot 2^{m})$ 
is independent of $v$ with the choice of $k_{(v)} := k_v = \Q_2\,k$ 
for all $v\,\vert\, 2$. 

To determine $\ov\No_{K/\Q} \,(\alpha)$ we note that 
$\ov\No_{K/\Q} \, (\alpha) = (-1)^\rho \,\No_{K/\Q} \,(\alpha)$, 
where $\rho$ is the number of conjugates of $\alpha$ 
which are negative in the real embeddings of $K$; 
from [Ma, \S\,3], the numbers $\rho$ and $c$ have same parity. 

Thus we have obtained in general: 

\begin{theo}\label{thmain} 
Let $L/K$ be a finite extension of number fields, 
unramified at~$2$. Denote by ${\mathfrak d}_{L/K}$ 
the discriminant of $L/K$, by $c$ the number of complex places 
of $L$ which lie above a real place of $K$, and by 
$\ov \No_{K/\Q}\,({\mathfrak d}_{L/K})$ 
the absolute norm of ${\mathfrak d}_{L/K}$. 
Let $k$ be the maximal subfield of $\Q(\mu_{2^\infty})$ contained 
in $K$ and put $[k:\Q] =: 2^m$, $m\geq 0$. 

\smallskip
Then we have the congruence 
$\ (-1)^c\,\ov \No_{K/\Q}\,({\mathfrak d}_{L/K}) \equiv 1 \bmod 
(4\cdot 2^{m})$. 
\end{theo} 

\begin{remark} 
We have the following improvement in two particular circumstances: 
if $k = \Q^{(m)}$, $m\geq 0$, then under at least one 
of the following two conditions: 

\smallskip 
\ \ (i)  $2$ splits totally in $K(\sqrt {\alpha\,})/K$,\, 
\footnote{If $\alpha = x^2\,(1+ 4\,y)$, 
    $y \in K^\times$, $y$ $2$-integer, this condition is equivalent to 
    ${\rm Tr}_{\F_v/\F_2} (y) = 0$ for all $v\,\vert\,2$, 
    where ${\rm Tr}_{\F_v/\F_2}$ is 
    the absolute trace from the residue field $\F_v$ of $K$ 
    (see [Gr, I.6.3, Lemma]).} 

    \smallskip 
\ \  (ii) the indices of ramification of $v\,\vert\,2$ 
in $K/k$ are all even, 

\smallskip\noindent 
we obtain the congruence 
$\ (-1)^c\,.\,\ov \No_{K/\Q}\,({\mathfrak d}_{L/K}) \equiv 
1 \bmod (4\cdot 2^{m+1})$.\,
\footnote{Use the computation of $\ov\No_{K/\Q}\, ({\mathfrak c})^2$ at the end 
of Sections 3, then Theorem~\ref{thlocal} for the computation 
of $\,\No_{K_v/\Q_2}\, (\alpha)$ for $v\,\vert\,2$ 
in these particular cases.} 
\end{remark}


\begin{thebibliography}{99} 

\bibitem[Gr]{Gr} G. Gras, 
{\it Class Field Theory: from theory to practice}, 
SMM, Springer-Verlag, 2003;  second corrected printing: 2005. 
%
\bibitem[Ma]{Ma} J. Martinet, {\it Les discriminants quadratiques 
et la congruence de Stickelberger}, 
S\'em. Th\'eorie des Nombres, Bordeaux {\bf 1} (1989), 197--204. 
%
\bibitem[Pi]{Pi} S. Pisolkar, 
{\it Absolute norms of $p$-primary units}, 
Jour. de Th\'eorie des Nombres de Bordeaux {\bf 21} (2009), 733--740. 
\end{thebibliography}
\end{document}